\documentclass[english]{article}
\usepackage[T1]{fontenc}
\usepackage[latin9]{inputenc}
\usepackage{babel}
\usepackage{amsmath}
\usepackage{amsfonts}
\usepackage{amssymb}
\usepackage{amsthm}
\usepackage{graphicx}
\DeclareGraphicsExtensions{.pdf,.png,.jpg}

\newtheorem{prop}{Proposition}[subsection]

\newtheorem{cor}{Corollary}[subsection]
\newtheorem{lemma}{Lemma}[subsection]

\newtheorem*{theorem*}{Theorem}

\theoremstyle{remark}
\newtheorem{remark}{Remark}[subsection]

\theoremstyle{definition}
\newtheorem{defin}{Definition}[subsection]
\newtheorem{example}{Example}[subsection]

\begin{document}

\title{The Multi-variable Affine Index Polynomial}

\author{Nicolas Petit}

\maketitle
\begin{center}Oxford College of Emory University\\100 Emory Street, Oxford GA 30054\\ petitnicola@gmail.com\end{center}

\begin{abstract}
\noindent We define a multi-variable version of the Affine Index Polynomial for virtual links. This invariant reduces to the original Affine Index Polynomial in the case of virtual knots, and also generalizes the version for compatible virtual links recently developed by L. Kauffman. We prove that this invariant is a Vassiliev invariant of order one, and study what happens as we shift the coloring of one or more components.

\end{abstract}

\section{Introduction}

The Affine Index Polynomial, defined in \cite{affineindexpolynomial}, was later generalized to the case of compatible virtual links in \cite{virtualknotcobordismaffineindex}. 
The aim of this paper is to extend this notion to the case of non-compatible virtual links, using a definition that reduces to the above invariants (up to some simple adjustments) when restricting ourselves to compatible links or knots.
We will call this extension the Multi-variable Affine Index Polynomial.
We will also prove that the multi-variable affine index polynomial, and as a consequence the two older invariants, are all Vassiliev invariants of order one for virtual knots/links, and we will discuss how the invariants associated to different colorings are related to each other.

This paper is structured as follows: section \ref{background} contains a brief review of the necessary background, including Vassiliev invariants and the previous versions of the AIP. Section \ref{multivariableaip} presents the multi-variable affine index polynomial and states the main results related to it. Section \ref{changingstartingpoints} studies what happens to the polynomial as we change the coloring of the link. Finally, Section \ref{proofs} contains the proofs of the main propositions.

\section{Background}
\label{background}
\subsection{Virtual knots and Vassiliev invariants}

We will work in the virtual knot and link setting.
For the unfamiliar reader, a virtual knot is a knot diagram with two types of crossings, classical (in which one strand goes over another) and virtual (which are artifacts of the fact that the diagram is non-planar), modulo the classical and virtual Reidemeister moves. The types of crossings and the Reidemeister moves are pictured in Figs. \ref{threecrossingtypes}, \ref{RM}, \ref{virtualRM}.

Virtual links are knot diagrams with multiple components (typically linked with each other).

\begin{figure}
\centering
\includegraphics[scale=0.1]{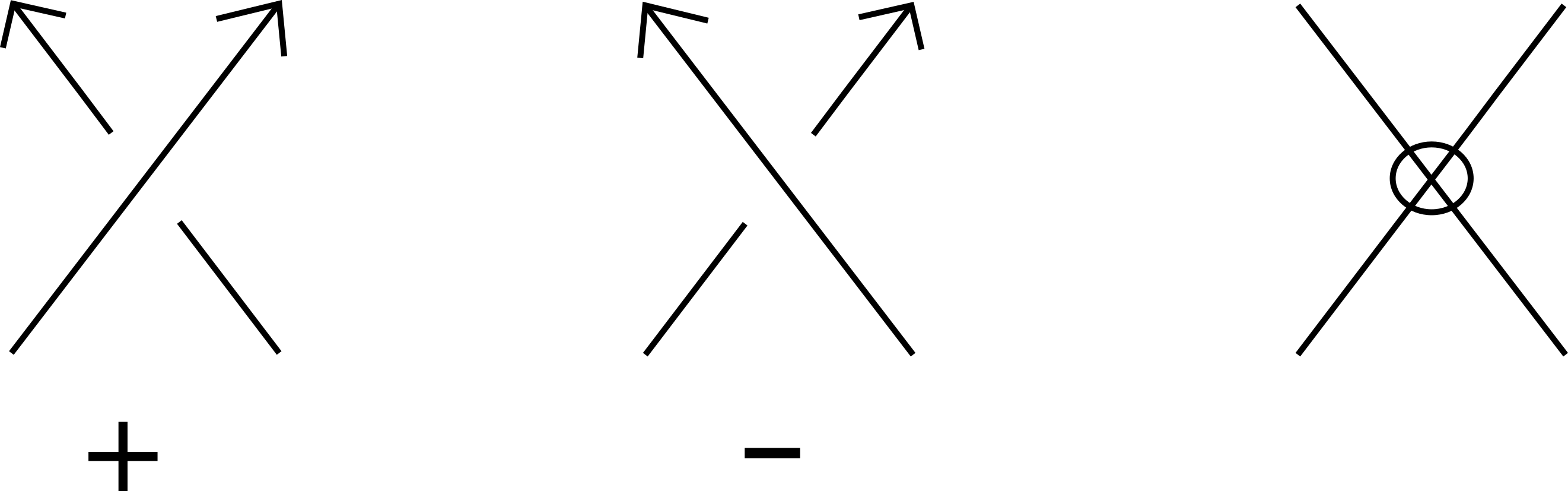}
\caption{The three types of crossing: positive, negative and virtual.}
\label{threecrossingtypes}
\end{figure}

\begin{figure}
\centering
\includegraphics[scale=.06]{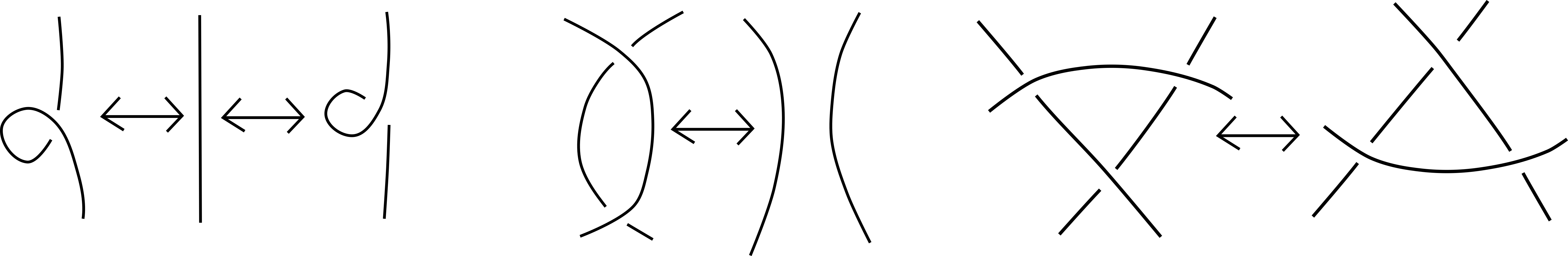}
\caption{The classical Reidemeister moves.}
\label{RM}
\end{figure}

\begin{figure}
\centering
\includegraphics[scale=.055]{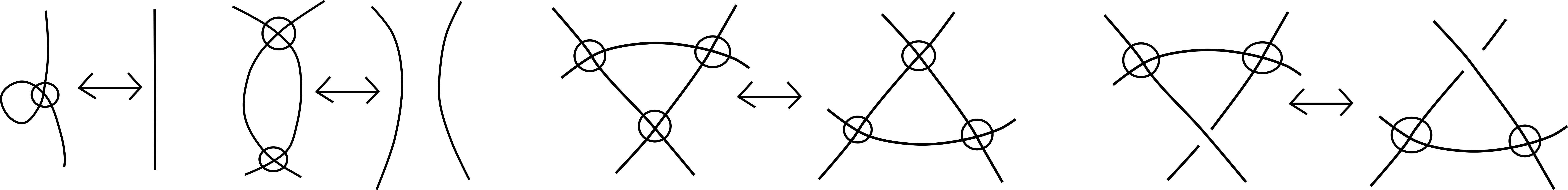}
\caption{The virtual Reidemeister moves.}
\label{virtualRM}
\end{figure}

Vassiliev invariants (or finite-type invariants) for virtual knots were introduced in \cite{virtualknottheory} as a straightforward generalization of the notion of finite-type invariant for classical knots. 
Given a virtual knot invariant $V$, we extend it to the category of virtual knots with double points by taking a weighted average of the two possible resolutions of the double point, as illustrated in Fig. \ref{doublepointresolution}.
We then say that $V$ is a finite-type invariant or Vassiliev invariant of order $\leq n$ if its extension identically vanishes on any virtual knot with more than $n$ double points.

\begin{figure}[!h]
\centering
\includegraphics[scale=0.06]{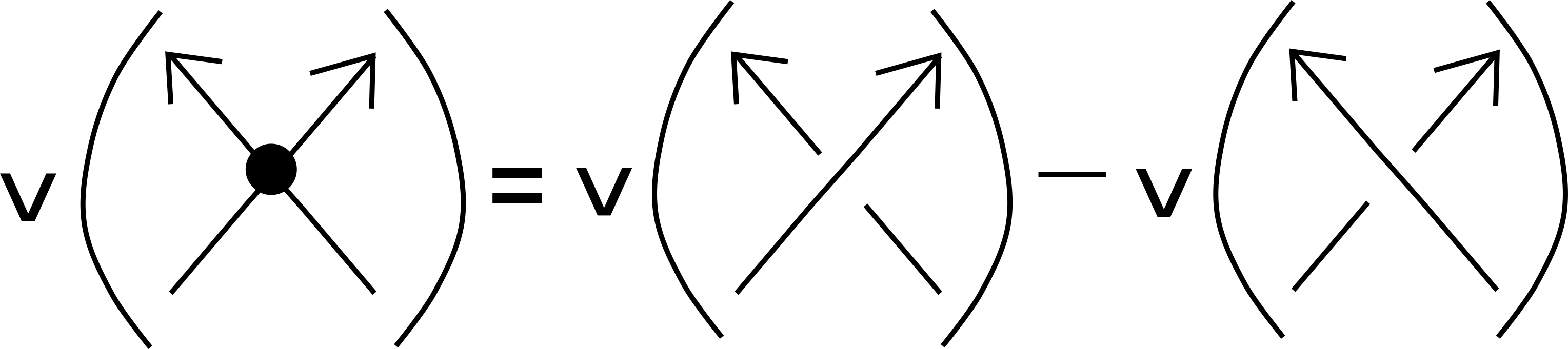}
\caption{How to extend a knot invariant $V$ to knots with double points. }
\label{doublepointresolution}
\end{figure}

\begin{remark}
Our reader should not confuse Vassiliev invariants with the more widely studied and classified collection of Goussarov-Polyak-Viro (GPV) finite-type invariants for virtual knots.
These are the extension of virtual knot invariants to the category of knots with a third type of crossings, called semi-virtual.
These semi-virtual crossings do not have a physical interpretation, and are defined as the difference between a classical crossing and a virtual crossing.
To compute a GPV invariant, one expands each semi-virtual crossing according to their definition and computes the invariant on the resulting combination of virtual knots.

GPV invariants are in some sense closer to the spirit of finite-type invariants for classical knots, as they can be characterized via the Polyak algebra and they can be computed via the use of arrow diagram formulae, similarly to how classical finite-type invariants are connected to the chord diagram algebra modulo the one- and four-term relations.
The interested reader can consult \cite{introtovki} and \cite{GPV} for more information on these invariants.
\end{remark}

Compared to the GPV invariants for virtual knots, little is known about the size or structure of Vassiliev invariants.
The first nontrivial GPV invariant for closed knots happens for $n=3$, while there are many interesting examples of Vassiliev invariants for $n=1,2$.
Examples of Vassiliev invariants include the coefficients of a power series expansion of both the Conway polynomial and the Kauffman Bracket polynomial \cite{virtualknottheory}, as well as various index polynomials (\cite{henrich}, \cite{longframedfti}, \cite{indexpolyvirtualtangles}) and the Three Loop Isotopy invariant of \cite{chrismandyethreeloop}.

\subsection{The Affine Index Polynomial}

This virtual knot invariant was originally introduced in \cite{affineindexpolynomial} and later extended to compatible virtual links in \cite{virtualknotcobordismaffineindex}. As our invariant generalizes both notions, we briefly recall them here.

\begin{defin}
\label{affinelabeling}
Let $K$ be a virtual knot. Label each arc of $K$ with an integer in the following way: pick an arc to have an arbitrary integer label, and propagate the label according to the rules of Fig. \ref{affinelabel}.

\begin{figure}
\centering
\includegraphics[scale=0.1]{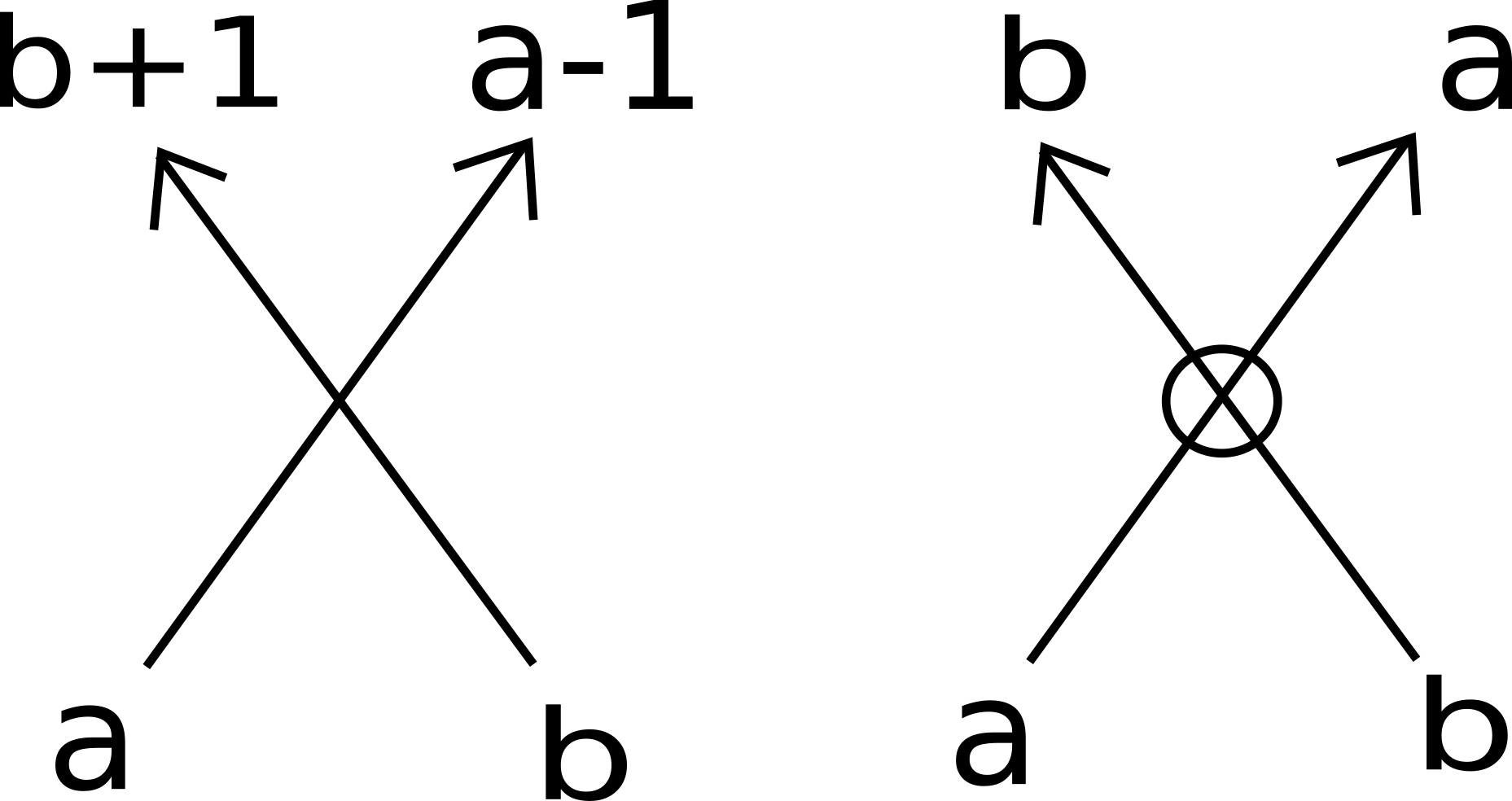}
\caption{How to propagate labels at crossings. The crossing on the left can be either positive or negative.}
\label{affinelabel}
\end{figure}

We call the resulting labeling of the knot $K$ an \emph{affine labeling} (or coloring) for $K$. We call the $\pm1$ shift in the label the \emph{index change} of the strand at the crossing: the bottom-left-to-top-right component has a $-1$ index change, while the bottom-right-to-top-left has a $+1$ index change.
\end{defin}

\begin{defin}\label{AIP}
Let $K$ be a virtual knot with an affine labeling $C$. Suppose that the generic crossing $c$ has labels as pictured in the left side of Fig. \ref{affinelabel}; if $c$ is a positive crossing we assign to it the weight $W_+(c)=a-(b+1)$, while if $c$ is negative we assign to it the weight $W_-(c)=b-(a-1)$.
We define an invariant by gathering the crossing weights into a polynomial $P_K(t)$ with the following two equivalent definitions:
$$P_K(t)=\sum_{c}sgn(c)(t^{W_{\pm}(c)}-1)=\sum_{c}sgn(c)t^{W_{\pm}(c)} - writhe(K),$$
where the sum is over all classical crossings of $K$. 
We call this the \emph{Affine Index Polynomial} of the knot $K$.\end{defin}

\begin{remark}
\label{remark1}
An easy mnemonic for remembering how to assign weights to each crossing: if the crossing is positive we look at the labels on the left side of the crossing, and take the bottom label minus the top label. If the crossing is negative we do the same with the labels on the right side instead.
\end{remark}

\begin{remark}
We will often use the acronym ``AIP'' for the Affine Index Polynomial.
\end{remark}

Note that the polynomial $P_K(t)$ is well-defined and independent of the starting label. If we picked a different starting integer label, all the labels on the knot would be shifted by the same fixed amount. 
Since the weights are defined as the difference between two labels, the shift gets added and subtracted, so the weight is left unchanged.

\begin{prop}
The Affine Index polynomial is a virtual knot invariant.
\end{prop}

The proof (found originally in \cite{affineindexpolynomial}) is just a matter of checking a generating set of Reidemeister moves.
The same paper also studies some properties of the Affine Index Polynomial, summarized in the following proposition.

\begin{prop}
Let $K$ be a virtual knot diagram, $\overline{K}$ be the result of changing the orientation of $K$, and $K^*$ be the mirror image of $K$.
We have that $P_{\overline{K}}(t)=P_K(t^{-1})$ and $P_{K^*}(t)=-P_K(t)$.
Finally, if $K$ is a classical knot, then $W_{\pm}(c)=0$ for all crossings, hence $P_K(t)=0$.
\end{prop}

The virtual link version of the Affine Index Polynomial was developed in \cite{virtualknotcobordismaffineindex}. We summarize the relevant content for our work in the following definitions and propositions. All proofs can be found in the above-mentioned paper.

\begin{defin}
A virtual link diagram is called \emph{compatible} if every component of the link has algebraic intersection number zero with the other components (in the sense of signed intersection numbers in the plane). A virtual link diagram admits an affine labeling if and only if it is compatible.
\end{defin}

\begin{defin}
Let $L$ be a compatible virtual link diagram. Define an affine labeling $C$ on $L$ by starting each component at an arbitrary integer value, typically represented by a lowercase letter, and propagating the label according to the rules of Fig. \ref{affinelabel}. 
Note that the rule makes no distinction between a self-crossing of a component, or a crossing involving multiple components (so-called external crossing).
Assign weights to every crossing in the usual way. Note that at self-crossings the resulting weight will be an integer not involving the arbitrary labels, while at external crossings the weight will include the difference of the two arbitrary labels.
Finally, construct the polynomial via 
$$P_L(t)=\sum_{c}sgn(c)(t^{W_{\pm}(c)}-1)=\sum_c sgn(c)t^{W_{\pm}(c)} - writhe(L)$$
\end{defin}

\begin{prop}
Given a virtual link $L$ and a coloring $C$ for $L$, the polynomial $P_L(t)$ is an invariant of $(L, C)$.
\end{prop}

We conclude this section with a simple computational example, repurpused from \cite{virtualknotcobordismaffineindex}; the interested reader can consult that paper to find more examples, as well as the necessary proofs.

\begin{example}
\label{example1}
Let's compute the AIP for the link in Fig. \ref{kauffmanexample}, which has already been colored with two arbitrary labels $a$ and $b$.

\begin{figure}[!h]
\centering
\includegraphics[scale=0.18]{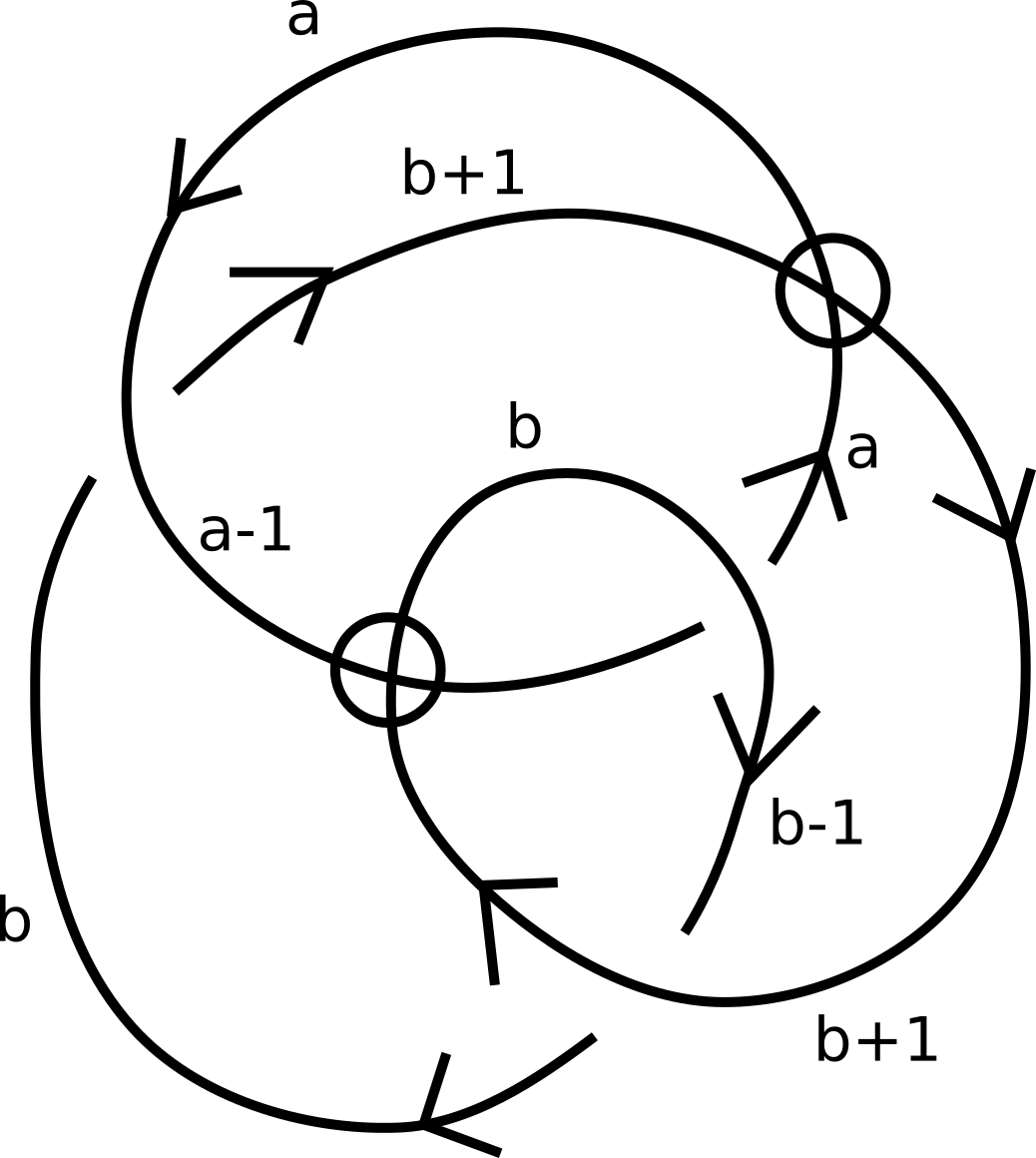}
\caption{A simple two-component link with an affine coloring.}
\label{kauffmanexample}
\end{figure}

To compute the AIP for this link, we look at the labels at the three classical crossings. Since all three crossings are positive, to compute the weight we look at the left labels (for the crossing in standard position) and subtract the top label from the bottom label, see Remark \ref{remark1}. 
At the self-crossing this yields $t^{b+1-b}=t$, at the top left crossing we get $t^{a-(b+1)}=t^{a-b-1}$ and at the top right crossing we get $t^{b-a}$.
Since the total writhe is $3$, we finally get
\begin{equation}\label{polykauffmanexample}p_L(t)=t^{a-b-1}+t^{b-a}+t-3.\end{equation}
\end{example}

\begin{remark}
At any external crossing, the weight will contain the difference of the two arbitrary labels of the components involved. Out of convenience, a new set of arbitrary variables can be introduced to somewhat simplify the resulting expression. In \cite{virtualknotcobordismaffineindex}, the author sets $a-b=N$ and writes the polynomial as $p_L(t)=t^{N-1}+t^{-N}+t-3$ instead.
\end{remark}

\section{The Multi-variable Affine Index Polynomial}

\subsection{Definitions of the invariant}
\label{multivariableaip}
\begin{defin}
Let $L$ be a compatible, oriented virtual link, whose components $L_i$ are ordered. 
For each component, pick a starting point and assign to it a bilabel containing two arbitrary variables, picking different variables for each component. 
We will typically use letters of the alphabet, so the first component will start its label at $(a_1, a_2)$, the second component at $(b_1, b_2)$, etc.

Now propagate the labels according to the rules of Fig. \ref{affinebilabeling}. 
Because the link is compatible, the labeling will wrap around and end at the same value it started.
We call the resulting coloring an \emph{affine bilabeling} for the compatible virtual link $L$.
\begin{figure}[h]
\centering
\includegraphics[scale=0.1]{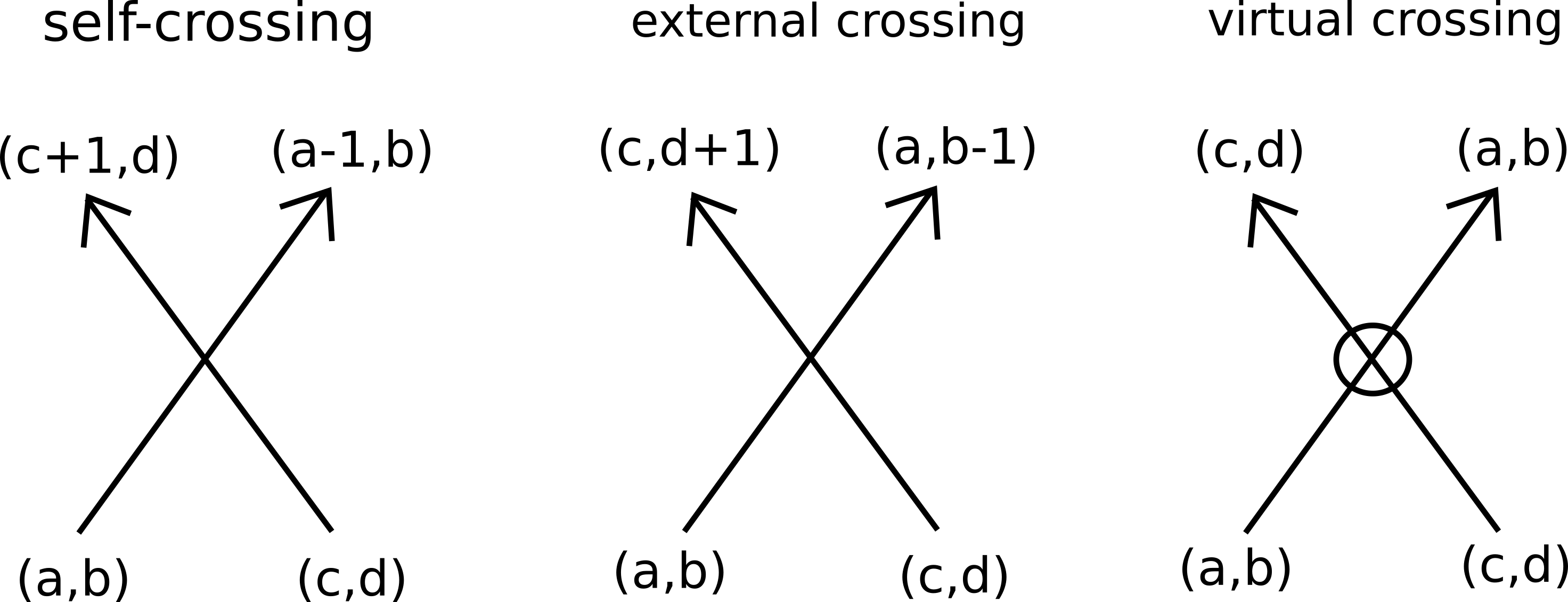}
\caption{How to propagate the bilabel. Here $a,b,c,d$ are arbitrary integers.}
\label{affinebilabeling}
\end{figure}

\end{defin}

\begin{defin}
Let $L$ be a compatible, oriented virtual link, with $n$ ordered components $L_1, \ldots, L_n$, and $C$ an affine bilabeling for it. Assign to component $L_i$ the variable $t_i$. 
Assign to each classical crossing $c$ a weight $W_{\pm}(c)$ in a similar manner to Definition \ref{AIP}: if the labeled crossing looks like the left picture of Fig. \ref{affinebilabeling}, assign to a positive crossing the weight $W_+(c)=a-(c+1)+b-d$, and to a negative crossing the weight $W_-(c)=c-(a-1)+d-b$.
Now construct the polynomial with the formula
$$p_L(t_1, \ldots, t_n)=\sum_{i=1}^n\ \sum_{c\in T_i\text{ over }T_j}sgn(c)(t_i^{W_{\pm}(c)}-1)$$
where the second sum is over all classical crossings where component $T_i$ goes over another component, itself included, and the index $i$ in the variable $t_i$ matches the component of the overstrand.
We call the resulting polynomial the Multi-variable Affine Index Polynomial of $L$.
\end{defin}

\begin{remark}
\label{labelingconvention}
Again, an easy mnemonic to remember the weights is to look at the left/right side of the crossing and subtract the top label from the bottom, then adding the two components of the bilabel together.
Also, we will use uppercase letters as a shorthand for the sum of the associated lowercase variables, e.g. $A=a_1+a_2$, $B=b_1+b_2$, etc.
\end{remark}

\begin{prop}
\label{prop1}
The Multi-variable Affine Index Polynomial is an invariant of the pair $(L, C)$, where $L$ is a compatible virtual link and $C$ is an affine bilabeling as described above. This invariant generalizes the AIP from \cite{affineindexpolynomial} and \cite{virtualknotcobordismaffineindex}.
\end{prop}

We will go over the proof (and all other proofs) in the section \ref{proofs}.
\begin{example}
As an example, let's consider the same link from Fig. \ref{kauffmanexample} and compute its multi-variable AIP. The same link with its new coloring is pictured in Fig. \ref{kauffmanbilabel}. Ordering of components now matters, so we will take the component with the $a$ labels as component one, and the other as component two.
\begin{figure}[!h]
\centering
\includegraphics[scale=0.2]{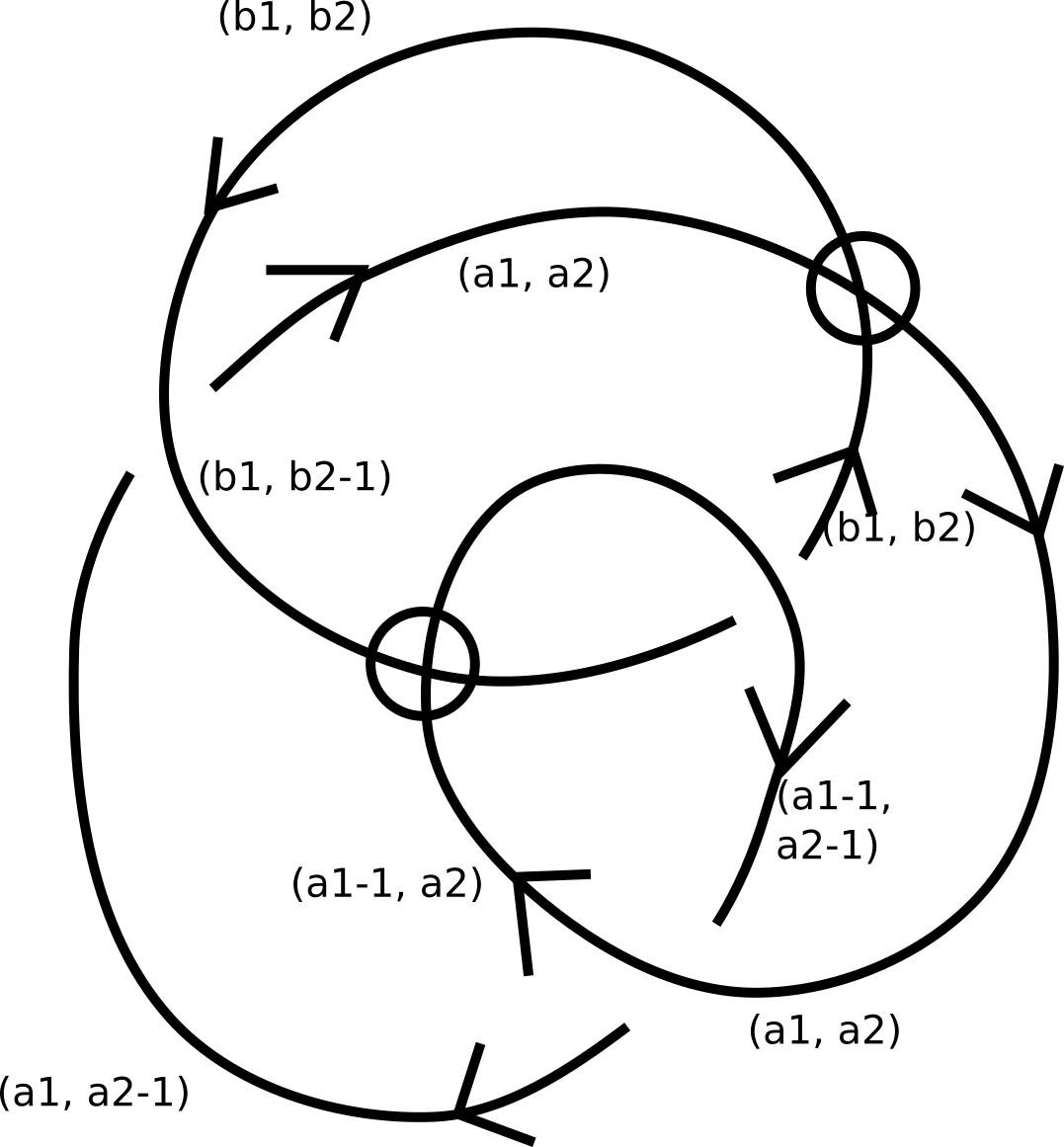}
\caption{The same link from Fig. \ref{kauffmanexample} with a bilabel coloring.}
\label{kauffmanbilabel}
\end{figure}
As before, we compute the polynomial by determining the weights at every classical crossing. The link is the same as the one in Example \ref{example1}, so all the crossings are positive, which means we need to look at the left labels at every crossing. At the self-crossing, the weight is $W_+=a_1-a_1+a_2-(a_2-1)=1$ and the overstrand belongs to component one, so the crossing contributes $+t_1^1$ to the polynomial.
At the top left crossing, the weight is $b_1-a_1+b_2-a_2=B-A$ and the overstrand belongs to component $2$, so the crossing contributes $+t_2^{B-A}$ to the polynomial.
Finally, at the top right crossing the weight is $a_1-1-b_1+a_2-b_2=A-B-1$ and component $1$ is the overstrand, so the crossing contributes $+t_1^{A-B-1}$.
Overall, the multi-variable affine index polynomial of the link with the given coloring is $$p_L(t_1, t_2)=t_1+t_1^{A-B-1}+t_2^{B-A}-3.$$
Note that setting $t_1=t_2$ and $A-B=N$ recovers the polynomial invariant from \cite{virtualknotcobordismaffineindex}, whose value is found in Eq. \ref{polykauffmanexample}.
\end{example}

\begin{remark}
Changing the ordering of the components will simply permute the variables, as the ordering has no influence on the value of the weights.
\end{remark}

\begin{prop}
\label{prop2}
The multi-variable AIP is an order one Vassiliev invariant of the pair $(L,C)$. As a consequence, the invariants from \cite{affineindexpolynomial} and \cite{virtualknotcobordismaffineindex} are also order one Vassiliev invariants.
\end{prop}

\begin{remark}
While we are not surprised that the original AIP is a Vassiliev invariant of order one, as so are many other index-type polynomials, we believe to be the first to explicitly state and prove that the (original) Affine Index Polynomial is a Vassiliev invariant of order one.
This is consistent with previous results \cite{indexpolyvirtualtangles}, where we showed that the Wriggle polynomial (which coincides with the AIP for virtual knots \cite{linkingnumberaffineindex}) is an order one Vassiliev invariant.
\end{remark}

Let us now extend this definition to non-compatible virtual links.
Start with an oriented, ordered, non-compatible virtual link $L$; to define a coloring on it, pick a starting point for each component, start the label at an arbitrary pair of variables, and propagate the labels using the same rules as before, pictured in Fig. \ref{affinebilabeling}.
When we get back to the starting point, we realize there is a slight issue: the first label (the one that tracks self-crossings) will correctly wrap around, but the second label (the one tracking external crossings) will end up as the starting label plus the intersection number of the component.
We then decide to ``discharge'' the excess intersection number on the starting point to make the labeling well-defined; we thus assign to each starting point the intersection number of that component, and call it the \emph{weight of the component}.
An example of a non-compatible link with such a coloring can be seen in Fig. \ref{noncompatexample}.

\begin{figure}[!h]
\centering
\includegraphics[scale=0.17]{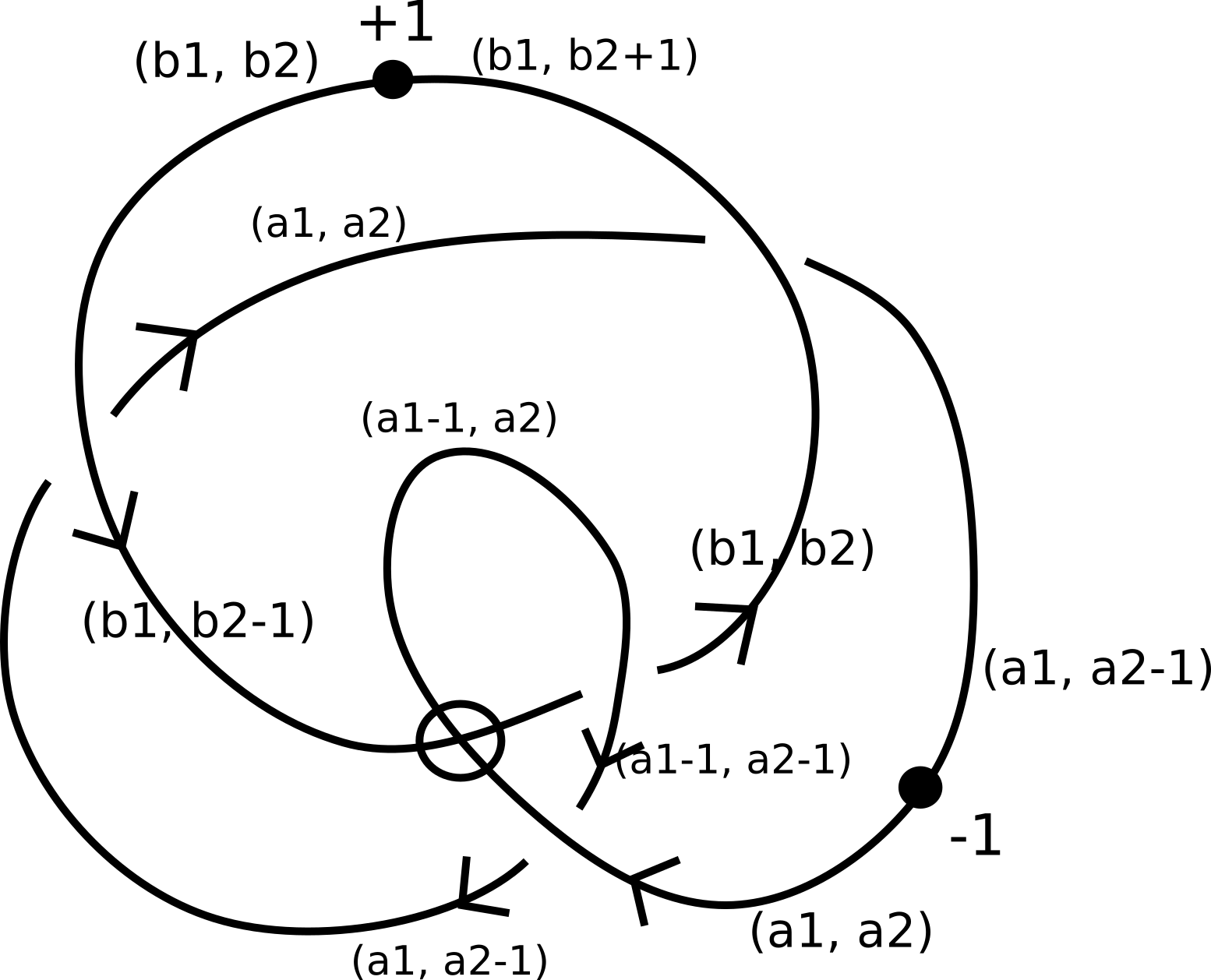}
\caption{A coloring for a non-compatible link. The starting points of our coloring are marked as thick dots, and the numbers next to them are the weights associated to the components.}
\label{noncompatexample}
\end{figure}

Once we have a coloring, we define the multi-variable AIP in the same exact way as before: compute the weight $W_{\pm}(c)$ of each classical crossing, take that weight as the exponent of the variable $t_i$ corresponding to the overstrand of the crossing, and take the signed sum over each crossing, subtracting the writhe at the end.
We still call the resulting polynomial the multi-variable AIP of the non-compatible link $L$.

\begin{example}
Consider the example of Fig. \ref{noncompatexample}; as before, the component with the $a$ labels will be component one, and the component with the $b$ labels will be component two.
The link now has four crossings to consider: a self-crossing at the bottom, and three external crossings.
The self-crossing of component two has weight $1$, so it will contribute $+t_1$ to the polynomial.
As for the external crossings, starting from the one at the top right and going counterclockwise, their contributions are respectively $-t_2^{B-A+1}$, $+t_2^{B-A}$, $+t_1^{A-B-1}$, where we use the labeling convention described in Remark \ref{labelingconvention}.
Overall, this example has multi-variable AIP equal to $$p_L(t_1, t_2)=t_1+t_1^{A-B-1}+t_2^{B-A}-t_2^{B-A+1}-2$$
\end{example}

\begin{prop}
\label{prop3}
The multi-variable AIP is an order one Vassiliev invariant of the pair $(L,C)$, where $L$ is a virtual link and $C$ is the associated coloring.
\end{prop}

\begin{remark}
If the link is non-compatible, performing a Reidemeister move on an arc that contains a starting point with nonzero weight will change the value of the invariant. As such, we must restrict Reidemeister moves to arcs and crossings away from the starting point.
\end{remark}

\subsection{Changing the starting point}
\label{changingstartingpoints}

The multivariable AIP is an invariant of the pair $(L, C)$, where $L$ is a virtual link and $C$ is the associated coloring. 
However, it can be useful to understand how the invariant will change as we move the starting points along the components of the link; we will describe this behavior in this section.
First of all, note that a compatible link is simply a non-compatible link whose weights are all zero. We will thus focus on the non-compatible case, and describe the behavior on compatible links as a corollary.

Clearly, the key to understanding how the polynomial changes is to fully understand how the coloring and weights change as we move the starting point around the knot. 
Moreover, virtual crossings do not change the coloring, so for the rest of this section when saying ``crossing'' we implicitly mean ``classical crossing''.
Figures \ref{changingstartingpointpositive} and \ref{changingstartingpointnegative} pictorially represent this change.

\begin{figure}
\centering
\includegraphics[scale=0.1]{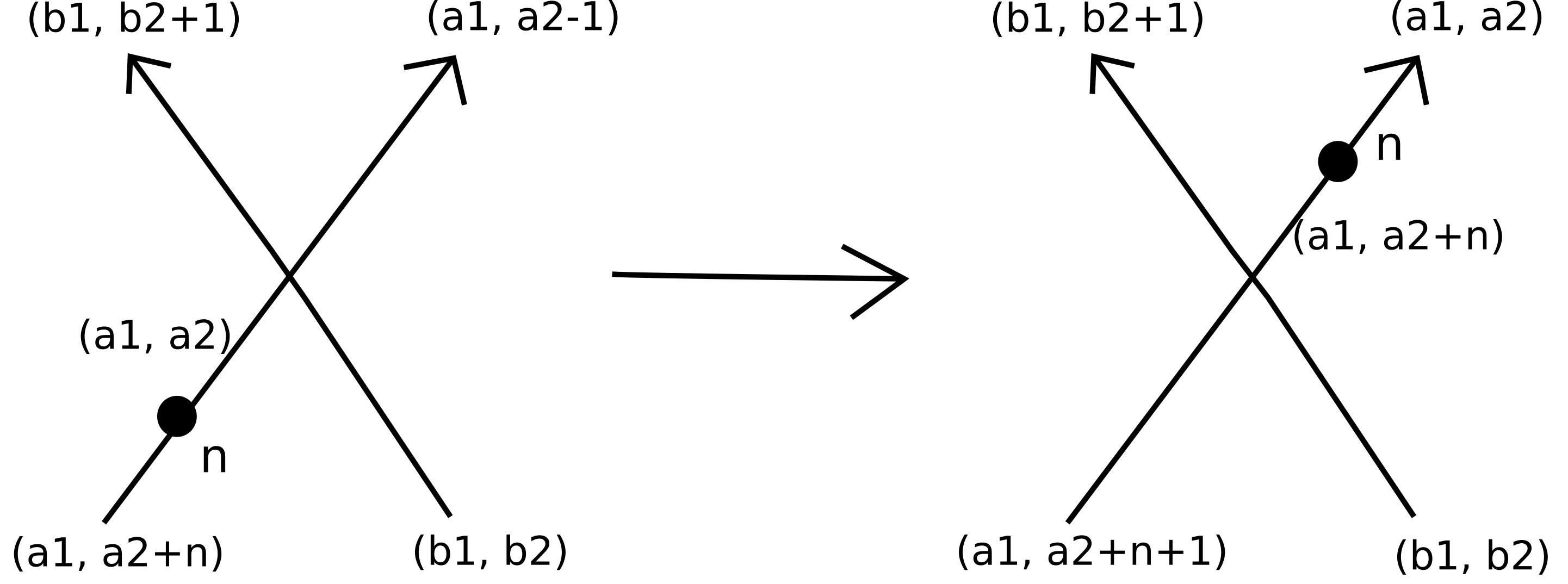}
\caption{How the bilabel changes as we move a starting point of weight $n$ through a crossing using the bottom-left-to-top-right strand.}
\label{changingstartingpointpositive}
\end{figure}

\begin{figure}
\centering
\includegraphics[scale=0.1]{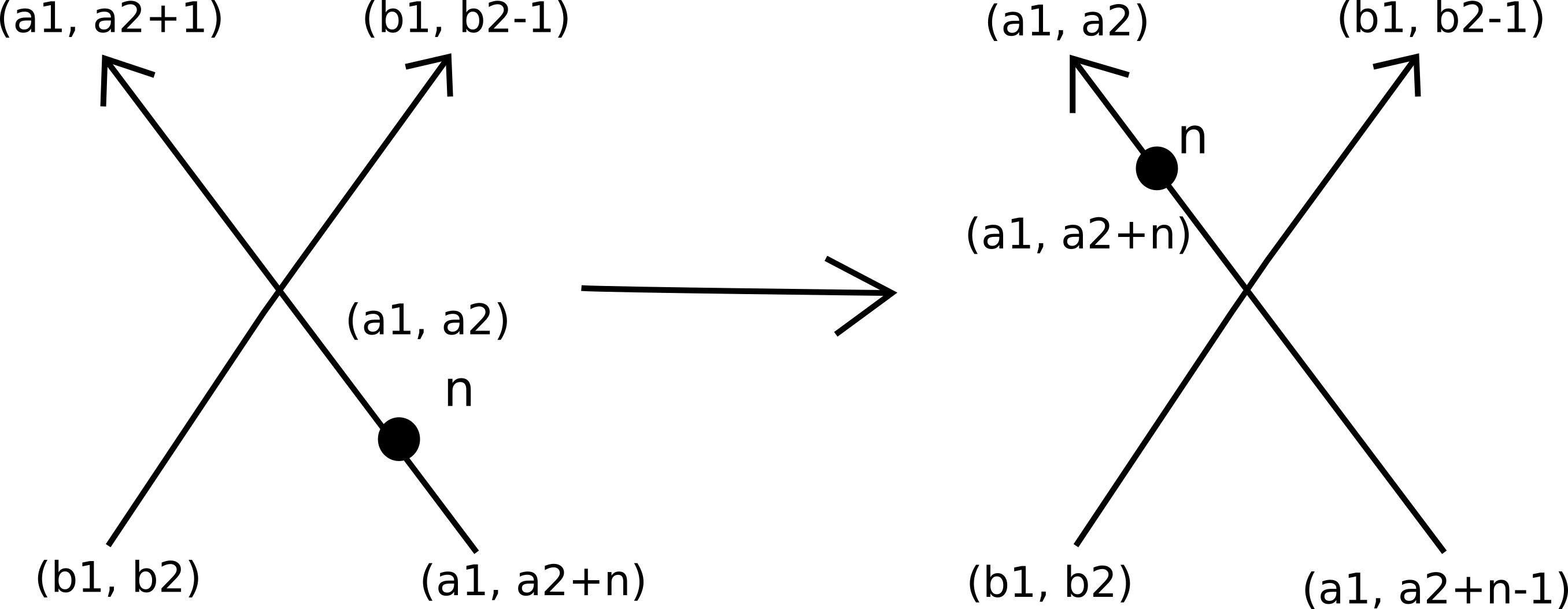}
\caption{How the bilabel changes as we move a starting point of weight $n$ through a crossing using the bottom-right-to-top-left strand.}
\label{changingstartingpointnegative}
\end{figure}

The labeling doesn't care for the sign of the crossing, but only for the direction of the strands at the crossing, so to study the bilabeling we can use flat crossings.
In Fig. \ref{changingstartingpointpositive}, we illustrate the change in the case the starting point belongs to a strand going through a $-1$ index change (see Definition \ref{affinelabeling} for a refresher on index change).
In that case, all the labels of the associated component are shifted by $+1$; this is easily explained by realizing that moving the starting point past the crossing means the postponing the $-1$ index change to the end of the circuit around the knot. 
Since the total index change stays the same, every other label needs to be increased by $1$.
In the other case, corresponding to the overpass of a negative crossing or the underpass of a positive crossing, and pictured in Fig. \ref{changingstartingpointnegative}, the labels of the associated component will all shift by $-1$.

Note that in the self-crossing case the shift will eventually propagate to the other strand involved in the crossing, so we get a picture like the one of Fig. \ref{changingstartingpointself}.
Also note that the weight of the component had no impact on how much we shifted the bilabel.
Finally, if we were to move the starting point all the way around the component, we would get back to our original label, as the total change in label is equal to the intersection number of the component.

\begin{figure}
\centering
\includegraphics[scale=0.1]{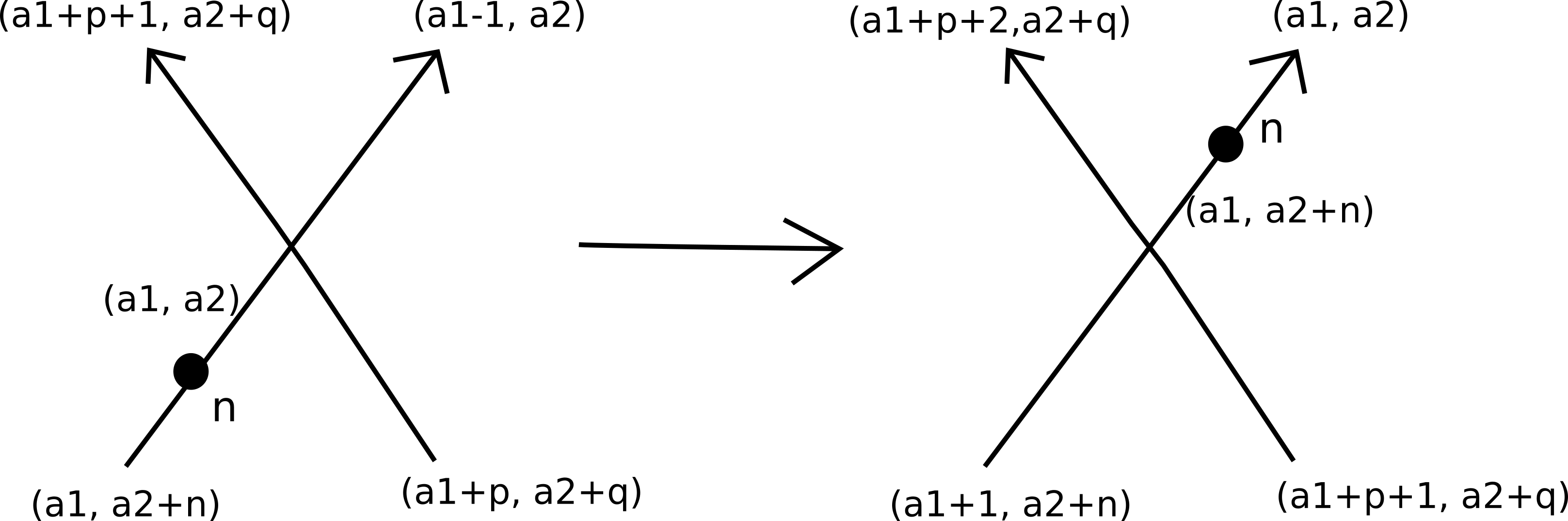}
\caption{An equivalent of Fig. \ref{changingstartingpointpositive} for the self-crossing case. Here $p, q$ are integers capturing the index change accumulated in returning to the self-crossing.}
\label{changingstartingpointself}
\end{figure}
\begin{figure}
\centering
\includegraphics[scale=0.1]{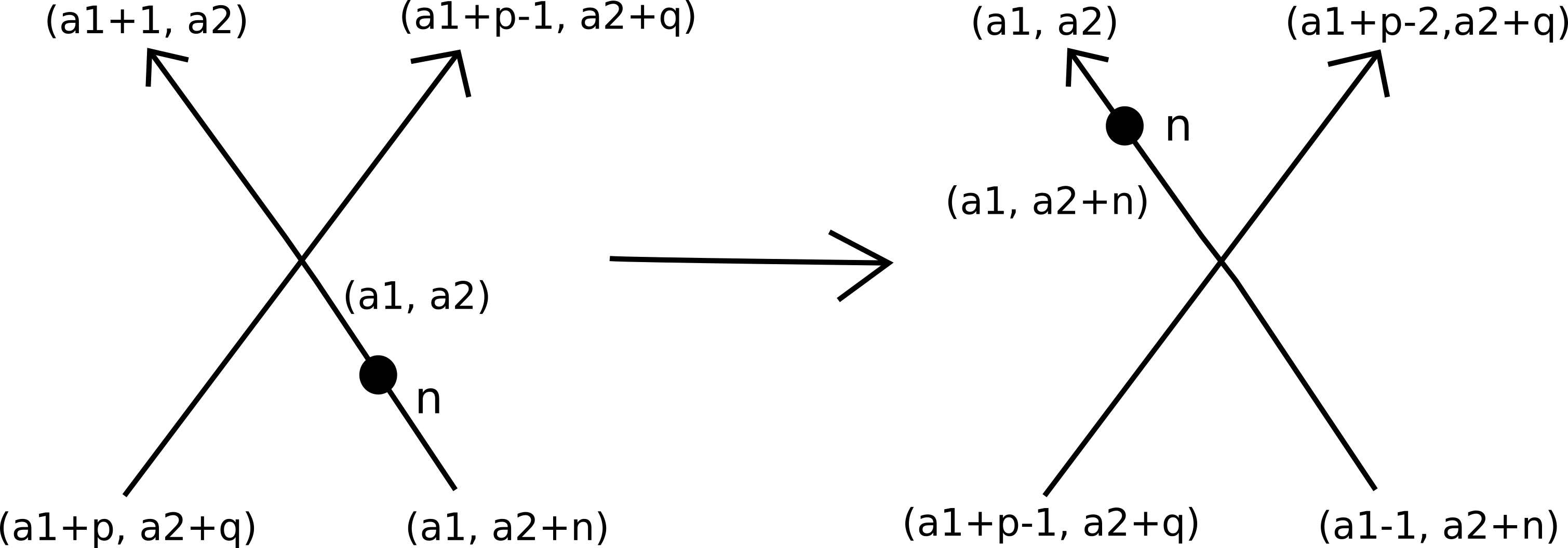}
\caption{An equivalent of Fig. \ref{changingstartingpointnegative} for the self-crossing case, with $p,q$ defined as in Fig. \ref{changingstartingpointself}.}
\label{changingstartingpointself2}
\end{figure}

Now that we understand how the bilabel changes, let's turn our attention to the weights of each crossing.
Let's start with Fig. \ref{changingstartingpointpositive}: we want to compare the weights of that crossing before and after the shift, for both positive and negative crossings.
If the crossing of Fig. \ref{changingstartingpointpositive} is positive, the weight of the left crossing is $A-B-1$ and the weight of the right crossing is $A-B+n+1-1=A-B-1+(n+1)$; if the crossing is negative, the weight of the left crossing is $B-A+1$ and the weight of the right crossing is $B-A-n=B-A+1-(n+1)$.
A similar computation for Fig. \ref{changingstartingpointnegative} shows that if the crossing is positive the weight goes from $B-A-1$ to $B-A-n=B-A-1-(n-1)$, and if the crossing is negative the weight goes from $A-B+1$ to $A-B+1+(n-1)$.

If we carefully organize the above information, we easily conclude the following lemma.
\begin{lemma}
\label{externalcrossingweightlemma}
If we move the starting point of a component of weight $n$ past an external crossing, the associated crossing weight will change by $\pm (n \mp 1)$, where the first $\pm$ is determined by whether we move the starting point through an overstrand ($+$) or understrand ($-$) of the crossing, and the $\mp 1$ is the opposite of the index change at that crossing for the given strand. All the other labels of the component will shift by $\mp1$ as well.
\end{lemma}

In the self-crossing case, we leave it to the reader to check the proof of the next lemma using a similar argument on Figs. \ref{changingstartingpointself} and \ref{changingstartingpointself2}.

\begin{lemma}
\label{selfcrossingweightlemma}
If we move the starting point of a component of weight $n$ past a self-crossing, the associated crossing weight will change by $-n$ if we moved the starting point past an undercrossing and by $+n$ if we moved the starting point past an overcrossing. 
All the other labels of the component will shift by $\mp1$, the opposite of the index change at that crossing.
\end{lemma}

Using the above two lemmas, we can finally determine how moving a starting point past a crossing changes the multi-variable AIP.

\begin{prop}
\label{prop4}
Let $L$ be an oriented, ordered virtual link, and $C$ an affine bilabel coloring for it. Suppose we move the starting point of component $i$, whose starting label is $(a_1, a_2)$ and component weight is $n$, past one crossing $c$ involving components $i$ and $j$ (where potentially $j=i$), and call this new coloring $C'$. We can obtain $p_{(L, C')}$ from $p_{(L, C)}$ by replacing the uppercase label $A=a_1+a_2$ with $A-i$, where $i$ is the index change of the strand at the given crossing, and multiplying the term in $p_{(L, C)}$ corresponding to the crossing $c$ by $t_i^{n}$ if we moved past an overcrossing or $t_j^{-n}$ if we moved past an undercrossing. 
\end{prop}

Having characterized how the invariant changes by moving a starting point past one crossing, we can relate the invariant on any two colorings $C, C'$ simply by moving the starting points one crossing and one component at a time, keeping track of the changes at each step.
We summarize this result in the next proposition

\begin{prop}
\label{prop5}
Let $L$ be an oriented, ordered virtual link, $C$ an affine bilabeling, and $C'$ a different affine bilabeling where we picked a different starting point for some or all of the components. 
For each component of $L$, track the crossings you pass in going from the starting point of $C$ to the starting point of $C'$; let $O_i$ be the set of all overcrossings of component $i$ you go through, and $U_i$ be the set of all undercrossings of component $i$ you go through. 
Finally, let $I_i$ be the total index change of component $i$ in going from the starting point of $L_i$ in $C$ to the starting point of $L_i$ in $C'$, and $n_i$ the weight of component $i$.

To obtain $p_{(L, C')}$ from $p_{(L, C)}$, subtract the total index change of each component from the relative uppercase label, and multiply the terms corresponding to the crossings of the sets $O_i, U_i$ by $t_i^{-n_i}$ for every crossing in $O_i$ and by $t_j^{n_i}$ for every crossing in $U_i$, where $j$ is the other component involved in the crossing (which could also be equal to $i$).
\end{prop}

\begin{example}
In Fig. \ref{noncompatexamplec} you can see a different choice of coloring for the same link as in Fig. \ref{noncompatexample}. The original colored link had invariant equal to $$p_{(L, C)}(t_1, t_2)=t_1+t_1^{A-B-1}+t_2^{B-A}-t_2^{B-A+1}-2.$$

\begin{figure}
\centering
\includegraphics[scale=0.17]{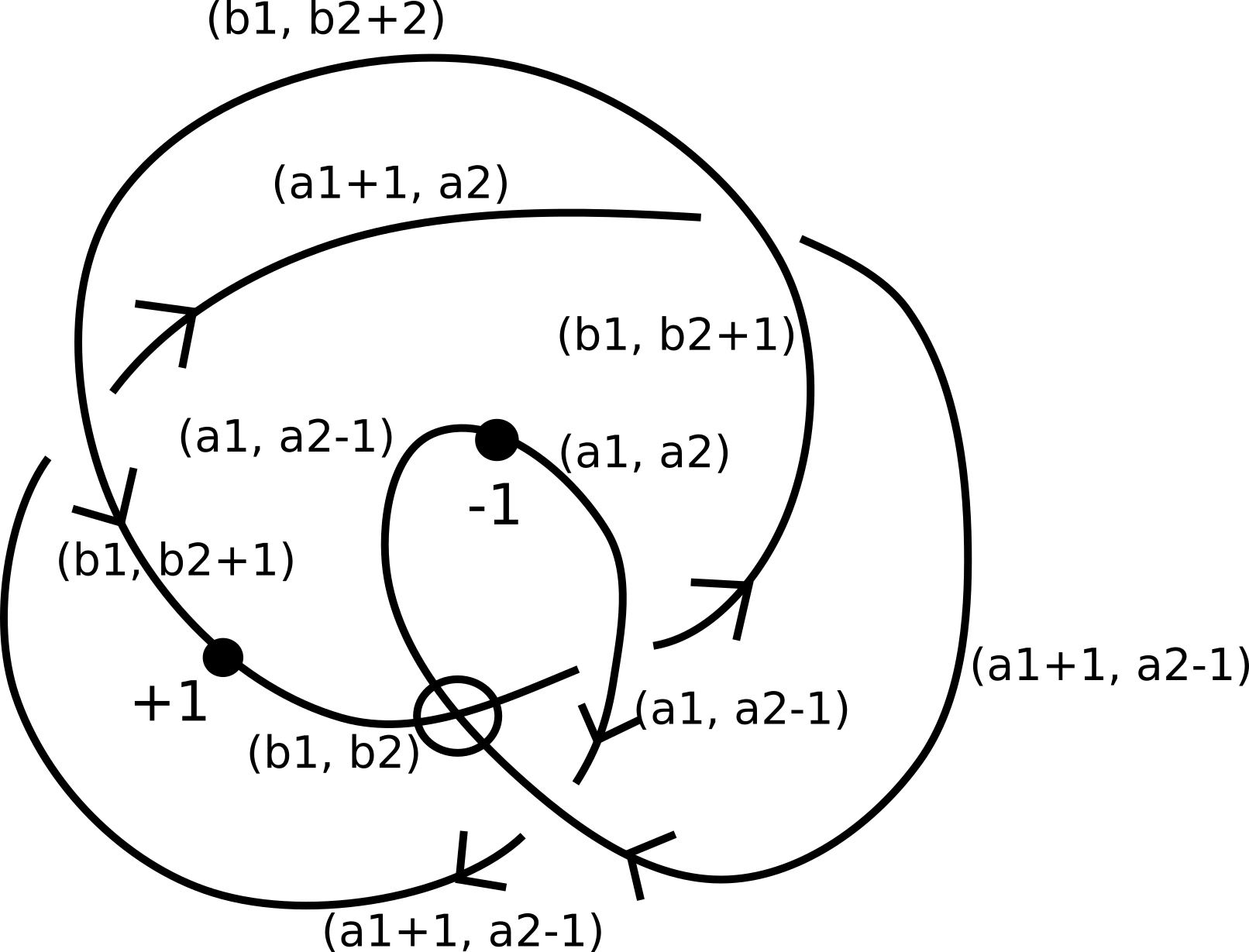}
\caption{Another coloring $C'$ for the link of Fig. \ref{noncompatexample}.}
\label{noncompatexamplec}
\end{figure}

To get from Fig. \ref{noncompatexample} to Fig. \ref{noncompatexamplec}, we moved the starting point of component one, whose weight is $n_1=-1$, past one crossing with index change $I_1=-1$; we went through the overcrossing of said crossing. We also moved the starting point of component two, whose weight is $n_2=+1$, past one crossing with index change $I_2=-1$; we also went through the overcrossing of said crossing.
According to our proposition, we can obtain $p_{(L, C')}$ by 
\begin{itemize}
\item replacing $A$ with $A-I_1=A+1$ and $B$ with $B-I_2=B+1$ in every term;
\item multiplying the term corresponding to the self-crossing (which was $t_1$) by $t_1^{n_1}=t_1^{-1}$;
\item multiplying the term corresponding to the top left crossing (which was $t_2^{B-A}$) by $t_2^{n_2}=t_2$.

\end{itemize}
This means we get
\begin{equation*}\begin{split}p_{(L, C'')}&=t_1^{-1}\cdot t_1+t_1^{(A+1)-(B+1)-1}+t_2\cdot t_2^{(B+1)-(A+1)}-t_2^{(B+1)-(A+1)+1}-2\\&=1+t_1^{A-B-1}+t_2^{B-A+1}-t_2^{B-A+1}-2.\end{split}\end{equation*}

A direct computation of the invariant from Fig. \ref{noncompatexamplec} gives the same result.
\end{example}

The above result is a bit unwieldy, because it requires us to keep track of what each crossing contributed to the multi-variable AIP. However, it helps us prove the following corollary.

\begin{cor}
Let $L$ be an ordered, oriented, compatible virtual link, and $C$ and $C'$ two distinct colorings for it. The self-crossing part of the multi-variable AIP (that is, the terms corresponding to the self-crossings of the various component) coincide for the colored links $(L,C)$ and $(L,C')$.\end{cor}

\begin{proof}
By Lemma \ref{selfcrossingweightlemma}, the weights at the self-crossings only change by $\pm n$, where $n$ is the weight of the component. Since the link is compatible every component has weight zero, so the weights of the self-crossings of the colored links $(L, C)$ and $(L, C')$ coincide.
\end{proof}

\subsection{Proofs}
\label{proofs}

\begin{proof}[Proof of Prop. \ref{prop1}] As usual, proving that the multi-variable AIP is a link invariant involves checking Reidemeister moves. We start by checking Reidemeister move one, pictured in Fig. \ref{r1affinebilabeling}.
A crossing involved in Reidemeister move one must be a self-crossing, so it is the first label that changes.
In both cases pictured in Fig. \ref{r1affinebilabeling}, the top and bottom labels on either side of the crossing are the same, so the weight of the crossing is zero. This means the kink contributes $\pm(t_i^0-1)=0$ to our polynomial.

\begin{figure}
\centering
\includegraphics[scale=0.17]{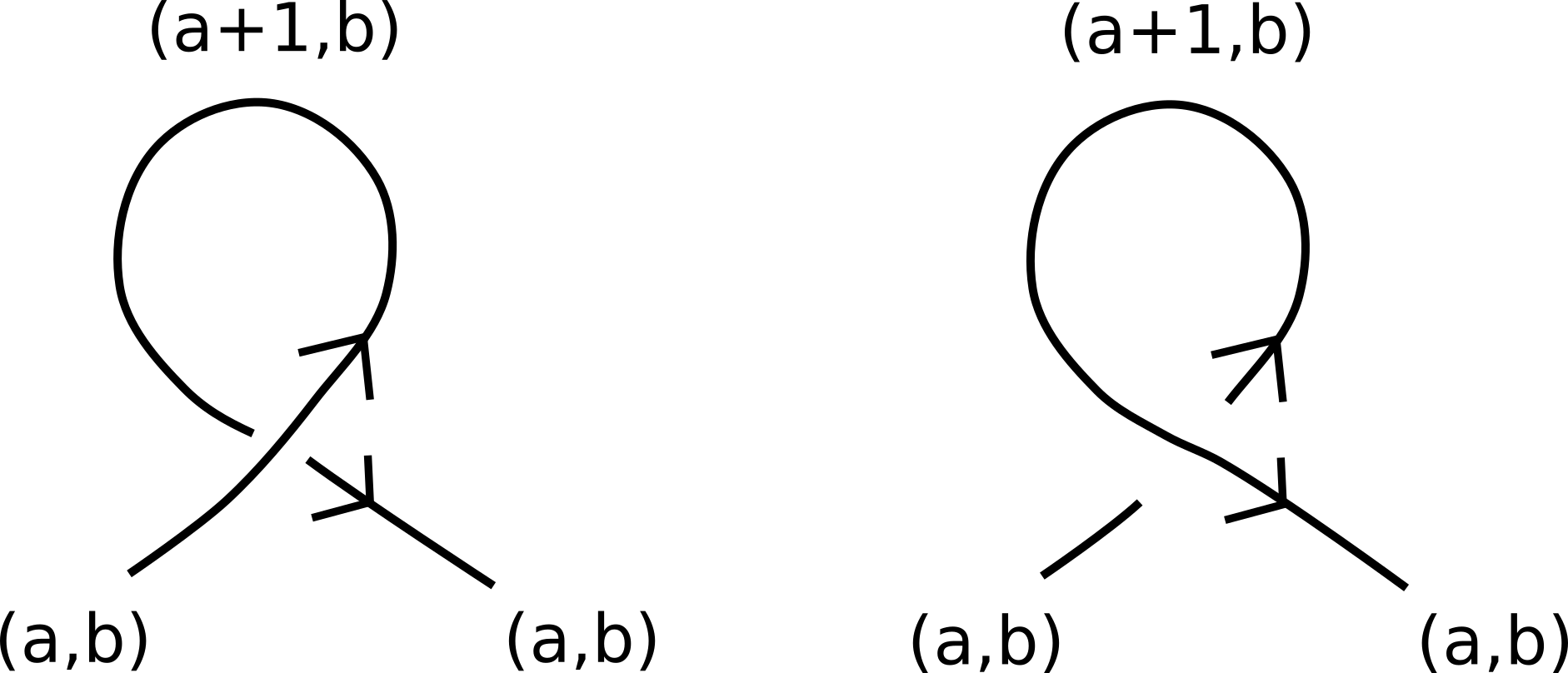}
\caption{The colored version of Reidemeister move one.}
\label{r1affinebilabeling}
\end{figure}

In the case of Reidemeister move two, we need to distinguish the case where both strands belong to the same component (and thus the crossings are self-crossings) from the case when the two strands belong to different components (so we have external crossings to consider).
If both strands belong to the same component, we get the situation pictured in Fig. \ref{r2affinebilabeling}.
The picture on the left has crossing weights of $W_+=a+b-(c+1)-d$ (bottom) and $W_-=(a-1)+b-c-d$ (top).
As these two weights are equal, the total contribution will be $+(t_i^{W_+}-1)-(t_i^{W_-}-1)=0$.
\begin{figure}
\centering
\includegraphics[scale=0.17]{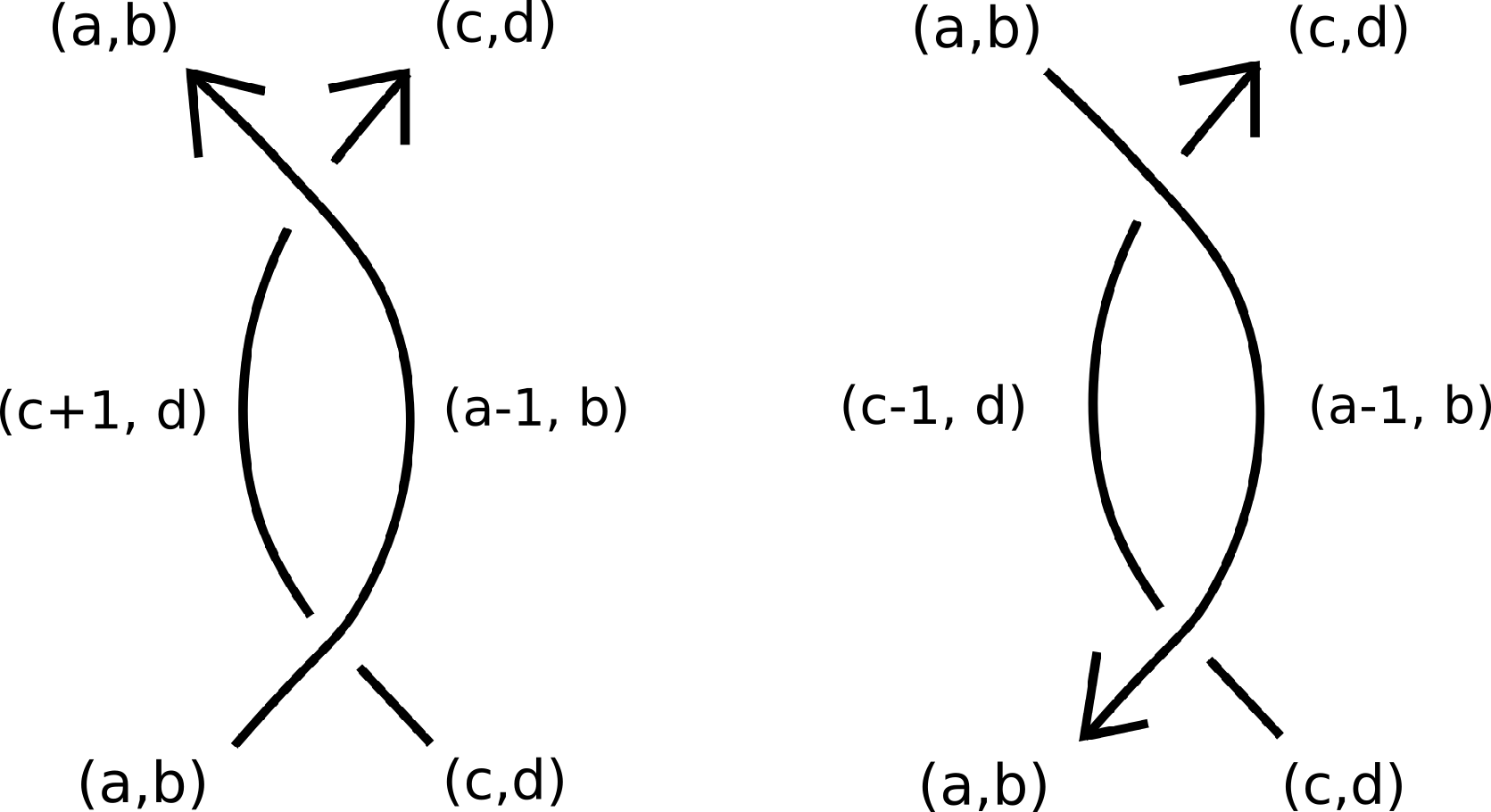}
\caption{The colored version of Reidemeister move two in the case both strands belong to the same component.}
\label{r2affinebilabeling}
\end{figure}
Similarly, the picture on the right has weights $W_+=a+b-c-d$ (top) and $W_-=(a-1)+b-(c-1)+d$ (bottom). As they are again equal, the total contribution of the right picture is also $+(t_i^{W_+}-1)-(t_i^{W_-}-1)=0$.
Since this is the same as the contribution when the two strands are separated (and there are no crossings), we have shown invariance under Reidemeister move two in the self-crossing case.

Let us now suppose that the two strands belong to different components, say the top strand is component $i$ and the bottom strand component $j$.
This situation is pictured in Fig. \ref{r2affinebilabelingmixed}.
Because of the way they're constructed, the weights are essentially unchanged from the self-crossing case; in the left picture, we have $W_+=a+b-c-(d+1)$ (bottom) and $W_-=a+(b-1)-c-d$ (top), while in the right picture we have $W_+=a+b-c-d$ (top) and $W_-=a+(b-1)-c-(d-1)$ (bottom).
\begin{figure}
\centering
\includegraphics[scale=0.17]{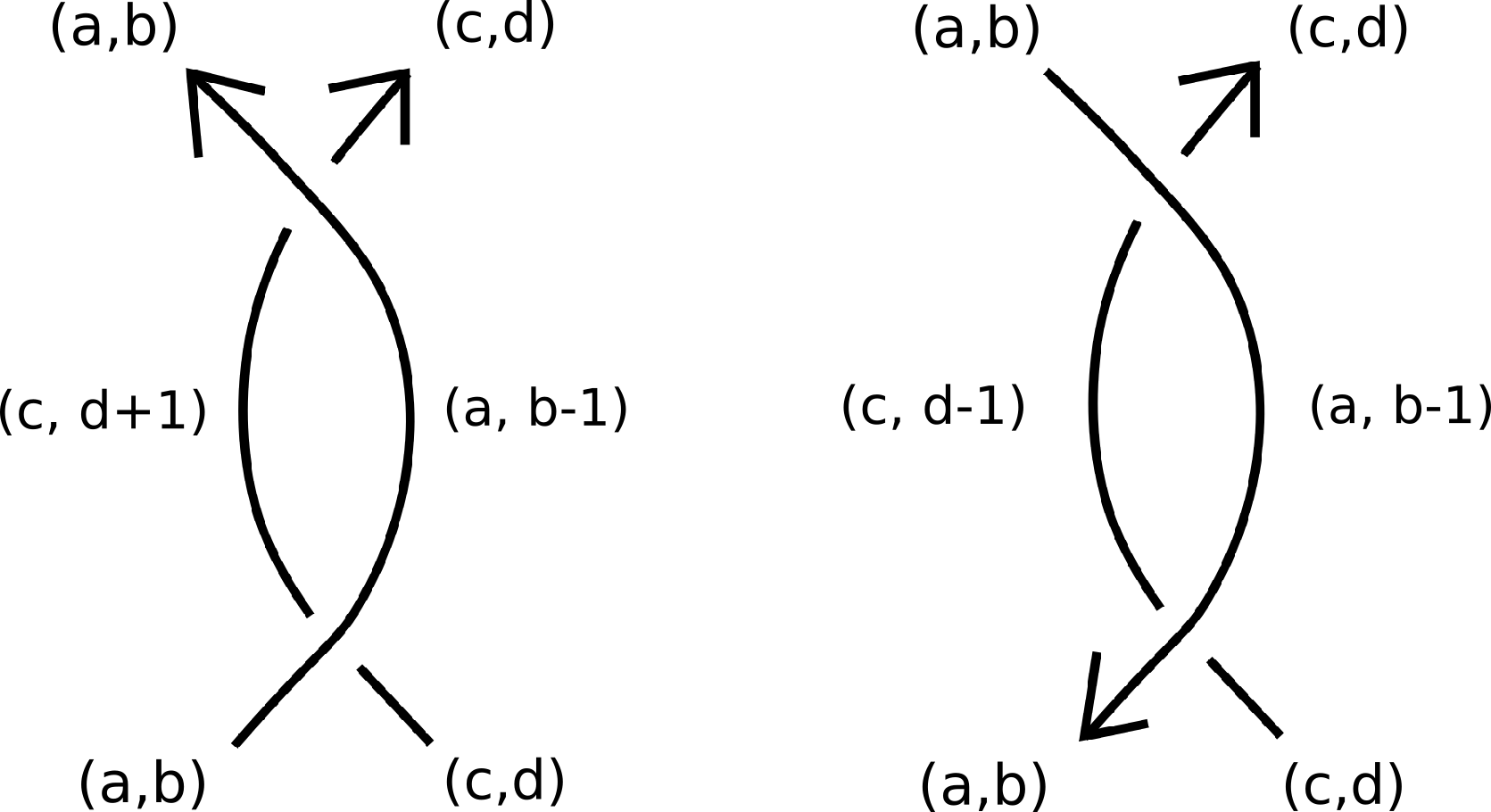}
\caption{The colored version of Reidemeister move two in the case the strands belong to different components.}
\label{r2affinebilabelingmixed}
\end{figure}
The weights are, once again, equal; moreover, the same strand goes over both crossings, so we will use the same variable (in this case, $t_i$) to express the contributing terms.
The overall contribution of either picture is then $+(t_i^{W_+}-1)-(t_i^{W_-}-1)=0$, which shows invariance under the second Reidemeister move when the strands belong to different components.

We need to consider two possibilities in the mixed Reidemeister move too; the component that forms two virtual crossings has no impact on the labels or the weights, but the other two strands can belong to the same component or different ones.
However, as Fig. \ref{rmixedaffinebilabeling} shows for the case of different components (the other case is analogous), the labels before and after the move coincide, so invariance under this move is trivially satisfied.
\begin{figure}
\centering
\includegraphics[scale=0.14]{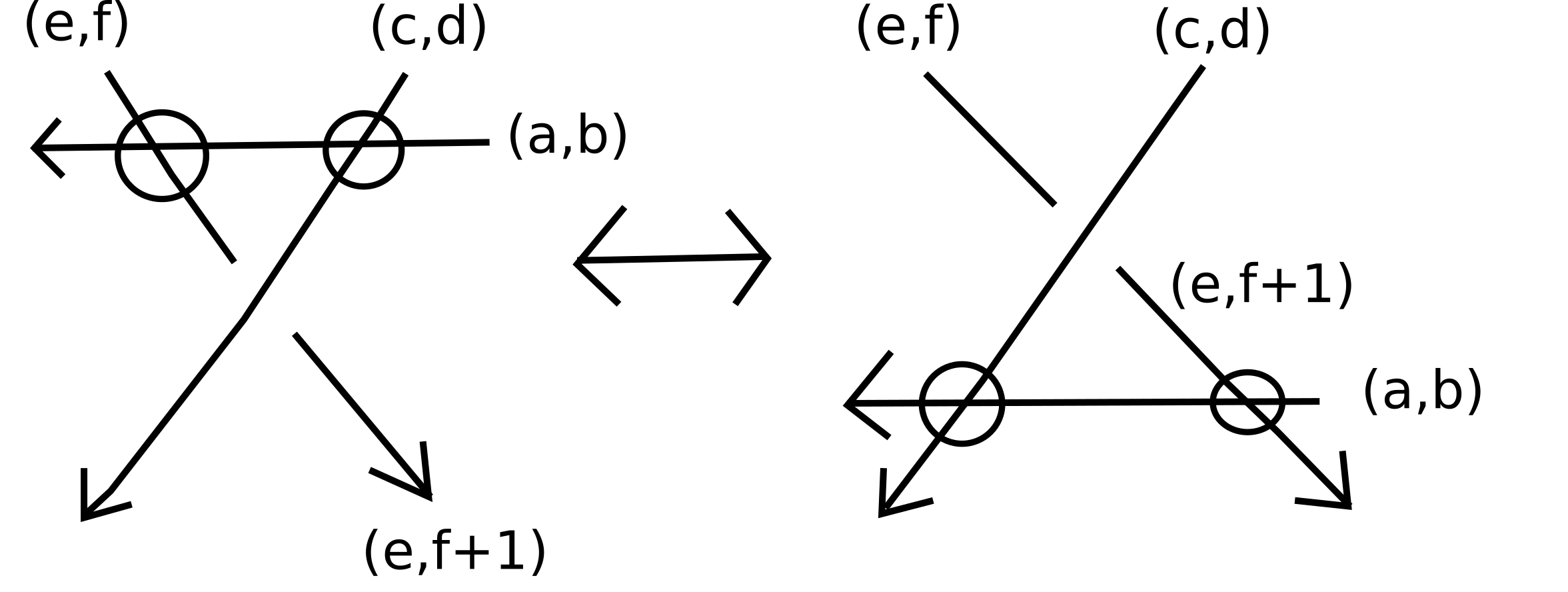}
\caption{The colored version of the mixed Reidemeister move.}
\label{rmixedaffinebilabeling}
\end{figure}

Finally, the case of Reidemeister move three.
Even given that our generating set of Reidemeister moves \cite{minimalgeneratingsetreidemeister} only use one Reidemeister move three, there are still many of cases to check, as each crossing could either be a self-crossing or an external crossing, depending on which component each of the three strands belongs to.
Since it's just a matter of checking every possibility (and the reader can rest assured we did our due diligence), we will only present one sample check and leave the rest as an exercise. 
\begin{figure}
\centering
\includegraphics[scale=0.14]{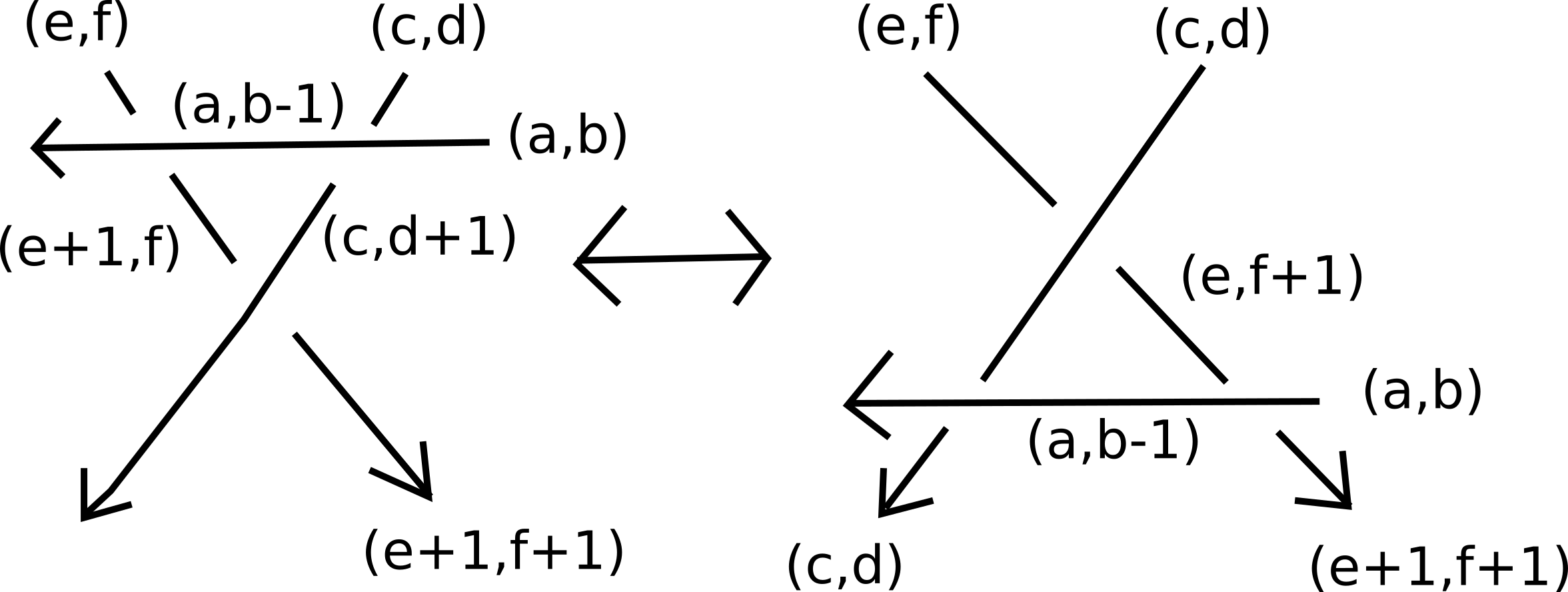}
\caption{One possible check for Reidemeister move three. Only the labels that are involved in computing weights are shown.}
\label{r3affinebilabeling}
\end{figure}
In our example, pictured in Fig. \ref{r3affinebilabeling}, the horizontal strand and the strand going from top left to bottom right belong to the same component, say component $i$, and the third strand belongs to a different component, say $j$.
As a result, in the left picture the top-left crossing is a self-crossing, while the other two crossings are external.
As usual with Reidemeister move three, there is a clear correspondence between the crossings on the left picture and those on the right picture, determined by which pair of strands is involved in a particular crossing.
Also note that this correspondence preserves the sign of the crossings, as well as which strand is the top strand of the crossing.
This means that we only need to check that the correspondence preserves the weights and we will have successfully proved invariance under Reidemeister move three.
The top-left crossing in the left picture has weight $a+(b-1)-(e+1)-f$, and corresponds to the bottom-right crossing in the right picture, whose weight is $a+b-(e+1)-(f+1)$; these weights are the same.
The top-right crossing in the left picture has weight $a+b-c-(d+1)$, and corresponds to the bottom-left crossing in the right picture, whose weight is $a+(b-1)-c-d$; these weights are the same.
Finally, the two middle crossings correspond to each other; the left one has weight $c+(d+1)-(e+1)-(f+1)$, and the right one $c+d-e-(f-1)$, so they are also equal.

Invariance under all Reidemeister moves has been checked, so the multi-variable AIP is in fact a virtual link invariant.
Because our bilabel separately tracks self-crossings and external crossings, if we collapse it so $(a_1, a_2)\mapsto a_1+a_2$ and we change the label at every crossing we encounter according to the usual rules, the resulting labels (and thus weights) are the same as the ones found in \cite{virtualknotcobordismaffineindex}. 
We can then easily reconstruct the generalization of the AIP found in said paper simply by setting all of our variables equal to each other, $t_i\mapsto t$ for all $i$.

In the case we have a virtual knot instead of a virtual link, all crossings are self-crossings, and in computing the weights the arbitrary labels will cancel out, so that the weight of each crossing is just the total change of index captured in the half-circuit of the knot from leaving the crossing to getting back to it the first time.
This is simply the definition of the original Affine Index Polynomial of \cite{affineindexpolynomial}, and concludes our proof of Prop. \ref{prop1}.

\end{proof}

\begin{proof}[Proof of Prop. \ref{prop2}]
Since we already proved that the quantity is an invariant, we just need to check that it vanishes on any virtual link with two double points, and that there is a virtual link with one double point on which the invariant is nonzero.
Let $L_{d,d'}$ be a link with two double points $d, d'$; after we resolve both crossing we get a sum of four alternating terms, as pictured in Fig. \ref{expansiondd2}. The two crossings in each term represent $d$ and $d'$, and every other crossing in the link is the same in all four pictures.
We claim that all the terms in the expression add up to zero. First of all, note that the labels must be the same in all four terms, as the labels are not influenced by whether a crossing is positive or negative.
This means that every crossing outside of $d, d'$ will have the same weight in all four terms, and contribute the same amount; but we are taking an alternating sum of said amounts, so these contributions will ultimately cancel out.
The situation for $d$ and $d'$ is similar: they too have the same labels in each of the pictures, and while their crossing weights could change when we swap between the positive and the negative version of the crossing, we can pair up terms. 
The terms with the same type of crossing (positive/negative) will have the same weight, and since they appear with opposite signs their contributions will also cancel out.

\begin{figure}
\centering
\includegraphics[scale=0.13]{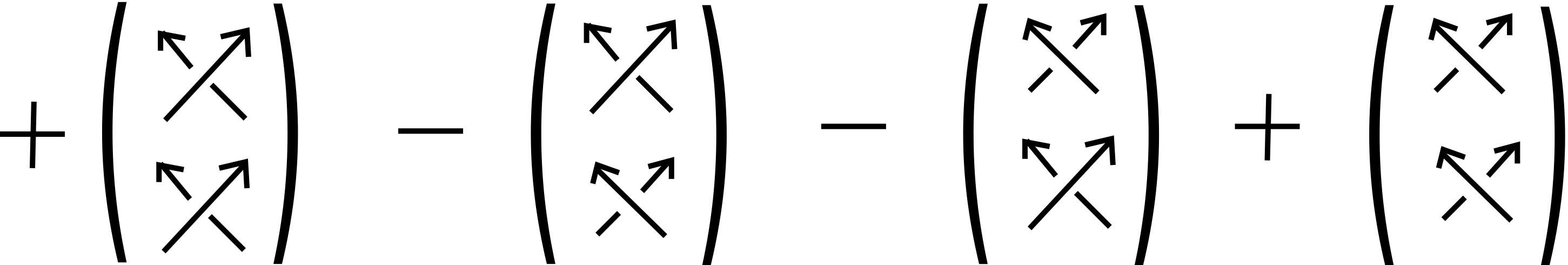}
\caption{A schematic representation of the expansion of the link $L_{d, d'}$. In each term, the top crossing represents $d$ and the bottom crossing $d'$.}
\label{expansiondd2}
\end{figure}

This shows that the multi-variable AIP is a Vassiliev invariant of order $\leq 2$.
Moreover, this proof rests on the fact that we can pair up crossings with the same weight and opposite sign; whether the crossing is a self-crossing or an external one has no bearing on it. 
We can then conclude that the same argument holds for the polynomial in \cite{virtualknotcobordismaffineindex} (whose crossing weights are ultimately equal to ours), and even for the polynomial in \cite{affineindexpolynomial}!

The simple link pictured in Fig. \ref{hopflinkdoublepoint} has a nonzero value, as different resolutions will lead to the unlink (whose invariant is zero) and the Hopf link (whose invariant under our chosen coloring is $t_1^{A-B-1}+t_2^{B+1-A}-2$).
It is easy to see that the same example will also work for the polynomial from \cite{virtualknotcobordismaffineindex} (we just set the two variables equal to each other, which still results in a nonzero expression).
Since Fig. \ref{hopflinkdoublepoint} is a link, that example does not work for the AIP for virtual knots; instead, we've provided the example in Fig. \ref{aipnonzero}, which already appears as an example in \cite{henrich}. Resolving the double point one way yields the unknot (whose polynomial is zero), while the other resolution has an AIP of $t^2+t^{-2}-2$.
This concludes the proof.

\begin{figure}
\centering
\includegraphics[scale=0.17]{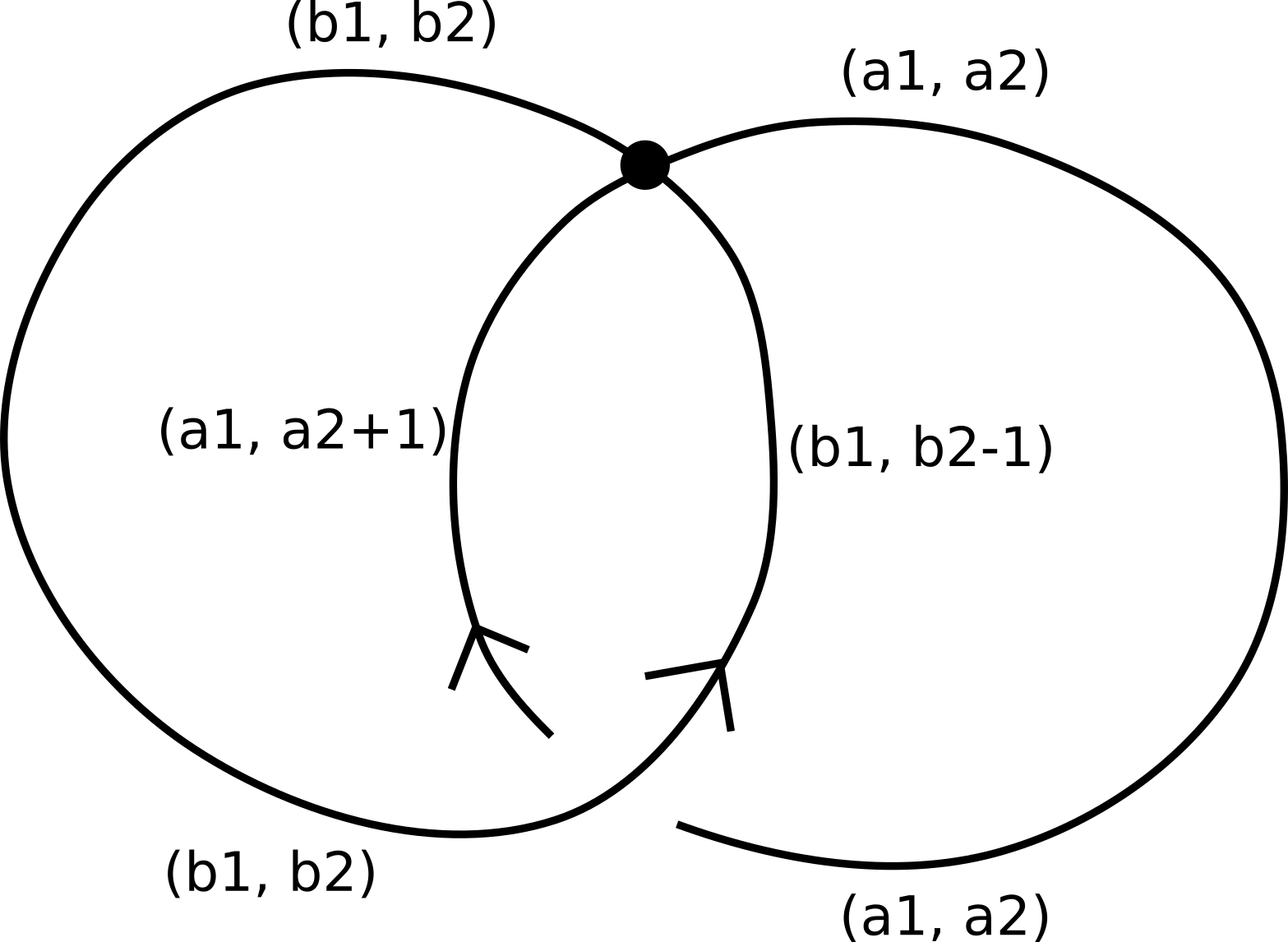}
\caption{A compatible virtual link with one double point whose multi-variable AIP is nonzero.}
\label{hopflinkdoublepoint}
\end{figure}

\begin{figure}
\centering
\includegraphics[scale=0.17]{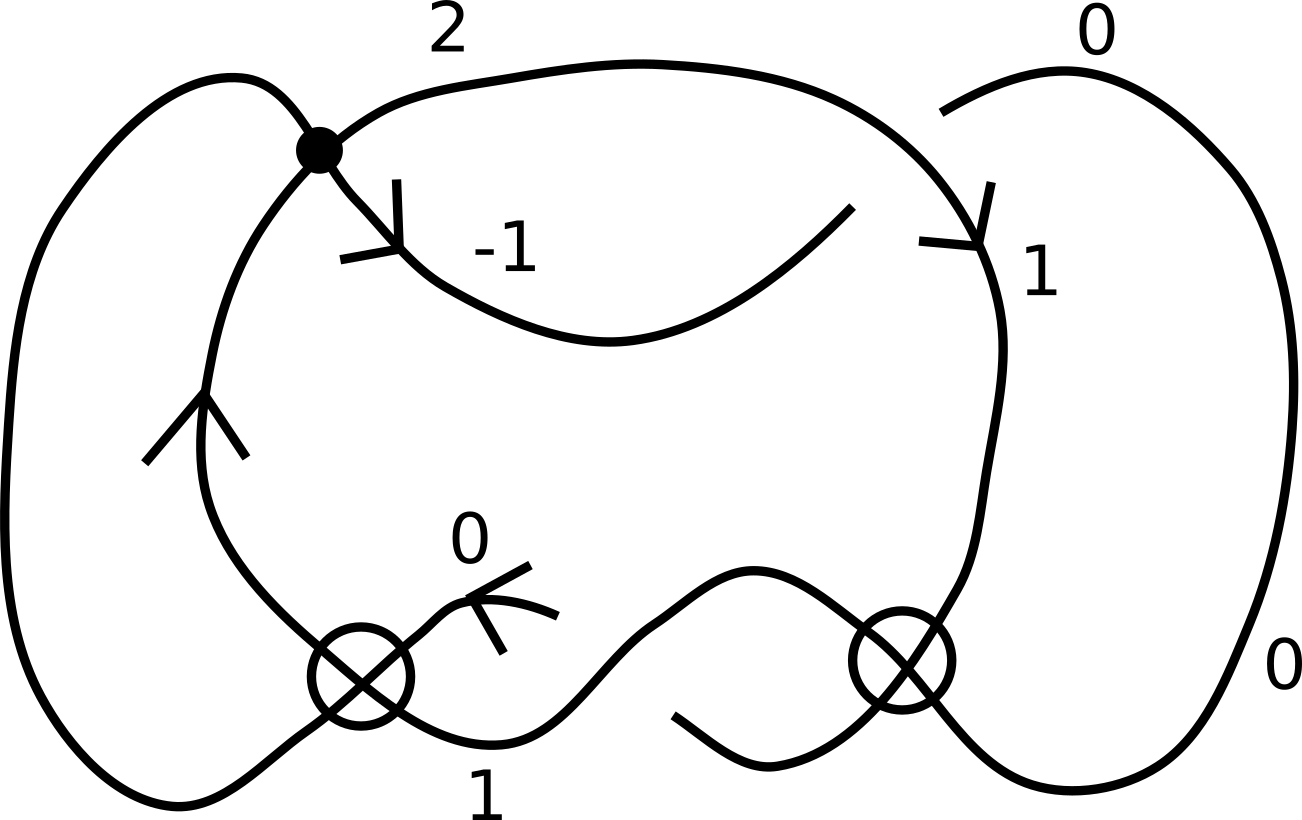}
\caption{A singular virtual knot on which the Affine Index Polynomial is nonzero.}
\label{aipnonzero}
\end{figure}

\end{proof}

\begin{proof}[Proof of Prop. \ref{prop3}]
Since we require the starting points of our coloring to be away from any Reidemeister moves, the proof that the quantity is an invariant is unchanged from Prop. \ref{prop1}.
Similarly, because the starting point of the coloring is fixed, we can apply the same exact argument of Prop. \ref{prop2} to show that the multi-variable AIP vanishes on non-compatible links with two double points.
Finally, compatible links are part of our definition (they simply have weight zero on each component), so the example of Fig. \ref{hopflinkdoublepoint} shows that the invariant is a Vassiliev invariant of order one.
\end{proof}

Proposition \ref{prop4} is a direct consequence of the two preceding lemmas, and Prop. \ref{prop5} follows from repeated applications of Prop. \ref{prop4}.

\section{Future work}

Work on this paper was started with the intention of developing a Vassiliev invariant for virtual tangles that generalized the AIP and was in some sense compatible with the stacking of virtual tangles (or, conversely, with the decomposition of virtual knots into a sum of virtual tangles).
As we started investigating how to extend the invariant from \cite{virtualknotcobordismaffineindex} to non-compatible links (the first step towards virtual tangles), we realized that the material was interesting enough to warrant its own paper.
We are thus currently in the process of extending these results to virtual tangles; preliminary results show that the multi-variable affine index polynomial naturally generalizes to the virtual tangle setting, and seems to satisfy the desired property regarding connected sum.
We are also studying some properties regarding orientation changes which we will present together with our virtual tangle results; the author hopes to have said paper available in a few months.

The author would like to acknowledge his institution, Oxford College of Emory University, for research support during the development of this paper.

\bibliographystyle{amsalpha}
\bibliography{fullbibliography}

\end{document}